\documentclass[a4paper]{amsart}

\usepackage{amssymb} 
\usepackage{amsmath}
\usepackage{amsthm}
\usepackage{dsfont, enumerate}
\usepackage{color}

\usepackage{verbatim}
\usepackage{url}
\newtheorem{thm}{Theorem}[section]
\newtheorem{lem}[thm]{Lemma}

\theoremstyle{definition}
\newtheorem{defn}[thm]{Definition}
\newtheorem{assump}[thm]{Assumption}
\newtheorem{examp}[thm]{Example}

\theoremstyle{remark}
\newtheorem{rem}[thm]{Remark}

\usepackage{caption}
\usepackage{array} 
\newcolumntype{M}[1]{>{\centering\arraybackslash}m{#1}} 

\newcommand{\rh}{0.09} 

\title{Weighted sieves with switching}

\author{Kaisa Matom\"aki}
\address{Department of Mathematics and Statistics, University of Turku, 20014 Turku, Finland}
\email{ksmato@utu.fi}

\author{Sebastian Zuniga-Alterman}
\address{Department of Mathematics and Statistics, University of Turku, 20014 Turku, Finland}
\email{szualt@utu.fi}

\begin{document}

\begin{abstract}
Weighted sieves are used to detect numbers with at most $S$ prime factors with $S \in \mathbb{N}$ as small as possible. When one studies problems with two variables in somewhat symmetric roles (such as Chen primes, that is primes $p$ such that $p+2$ has at most two prime factors), one can utilize the switching principle. Here we discuss how different sieve weights work in such a situation, concentrating in particular on detecting a prime along with a product of at most three primes. 

As applications, we improve on the works of Yang and Harman concerning Diophantine approximation with a prime and an almost prime, and prove that, in general, one can find a pair $(p, P_3)$ when both the original and the switched problem have level of distribution at least $0.267$.
\end{abstract}

\maketitle

\section{Introduction}
Classical sieve methods can be used to study integers without small prime factors. In particular, under certain hypotheses, they can give upper and lower bounds for 
\begin{equation}
\label{eq:SAzdef}
S(A, z) := \{n \in A \colon (n, P(z)) = 1\},
\end{equation}
where $A \subseteq \mathbb{N}$ and $P(z) = \prod_{p < z} p$. Among other things, this allows one to detect numbers that have only few prime factors --- if $A \subseteq [1, x] \cap \mathbb{N}$ and $S(A, x^{1/v}) > 0$, then $A$ necessarily contains numbers with at most $\lceil v-1 \rceil$ distinct prime factors. 

Kuhn~\cite{Kuhn41, Kuhn} first observed that if one attaches a weight $W(n)$ to every $n \in A$, with $W(n)= 1-\sum_{p \mid n} w_p$ and suitable coefficients $w_p \geq 0$, then one can find integers in $A$ with even fewer prime factors than with the above argument. Weighted sieves have then been developed by various authors, see e.g.~\cite[Chapters 9--10]{HalRich} and~\cite[Chapter 5]{GreavesBook} for more information. 

We shall reserve the letter $p$ with or without subscripts for primes and the notation $P_k$ for numbers with at most $k$ prime factors.

As an example, let us consider the twin prime conjecture. Take 
\[
A = \{p + 2 \colon p \text{ prime}, p \leq x\}.
\] 
Using the linear sieve (see Lemma~\ref{linearsieve} below), with level of distribution $1/2-\varepsilon/3$ coming from the Bombieri-Vinogradov theorem (for the definition of level of distribution, see Definition~\ref{def:lod} below), one obtains
\[
S(A, x^{1/4-\varepsilon}) > c_0 \frac{x}{\log^2 x}
\]
for some constant $c_0 = c_0(\varepsilon)$. This implies that $p+2 = P_4$ infinitely often. The level of distribution $1/2-\varepsilon/3$ is sufficient to replace $P_4$ by $P_3$ if one uses a weighted sieve with well-chosen weights $w_p$, see for instance~\cite[Theorem 9.2]{HalRich}.

Optimal weighted sieve coefficients are not known but there has been some work surrounding them, see for instance~\cite[Chapter 5]{GreavesBook}.

In 1973, Chen~\cite{Chen} made a breakthrough in the twin prime problem by showing that $p+2 = P_2$ infinitely often. Primes of this form are often called Chen primes. Chen had the ingenious idea of switching the roles of the variables in some parts of the argument. He first proved a lower bound for the number of primes such that $p+2 = P_3$ (with the prime factors of $P_3$ belonging to certain ranges). He then removed the contribution of products of three primes by applying an upper bound sieve to the sequence $p_1 p_2 p_3 -2$, thus sieving the prime $p$. Later Fouvry and Grupp~\cite{FouvryGrupp} managed to show that $p+2 = P_2$ infinitely often without using switching, but a better level of distribution coming from deep estimates for sums of Kloosterman sums.

There has been some interest in optimizing weights in Chen's set-up, see e.g. Wu~\cite{Wu1, Wu2}, which for instance leads to improved lower bounds for the number of Chen primes.

Chen's switching idea ~\cite{Chen} has also been utilized in finding pairs $(p, P_k)$ with $k \geq 3$ in various contexts like for instance the ones found in \cite{Vaughan}, \cite{Harman83} and \cite{LZX}. However, it seems like little attention has been paid to the corresponding choice of weights. The aim of this note is to initiate the study of this subject, and set up the scene for further investigations. Most of what we do works in a very general setting for detecting pairs $(p, P_k)$ using flexible weights. Nonetheless, we will in particular study the case $k = 3$ and illustrate the effect of some weights, which we intend to keep rather simple. Using other weights might lead to improvements on our results.

The following application demonstrates how one can obtain stronger results by simple changes to the sieve weights.
\begin{thm}
\label{th:HarmanImprove}
Let $\lambda_0 \in \mathbb{R}$ and let $\lambda_1, \lambda_2 \in \mathbb{R}\setminus\{0\}$ be such that $\lambda_1/\lambda_2$ is negative and irrational. Then there are infinitely many solutions to
\[
|\lambda_0 + \lambda_1 p + \lambda_2 P_3| < p^{-\rho}
\]
with $\rho = \rh$.
\end{thm}
This theorem improves on the work of Yang~\cite{Yang} that had exponent $\rho = 1/131$, which in turn improved on the work of Harman~\cite{Harman83} with $\rho = 1/300$. Here, we use the same arithmetic information as in those previous works, but with different sieve weights --- actually Harman~\cite{Harman83} wrote that $\rho$ could be improved ``by using better weights (for example see Section 2 of~\cite{HH-BR}) and making more use of Chen's technique''. What is perhaps surprising about Theorem~\ref{th:HarmanImprove} is the quality of the improvement and the fact that it is obtained using simpler sieve weights rather than more complicated ones. Unfortunately, despite the significant improvement on $\rho$, we are not able to replace $P_3$ by $P_2$ for any positive $\rho$, which would be the most desired result in this problem.

We will also show that if one has level of distribution $0.267$ in both the original and the switched problem, then one can find pairs $(p, P_3)$.  We will state this result formally only upon stating the needed assumptions precisely (see Theorem~\ref{th:constLOD} below). Note that in~\cite{LZX}, where Vaughan's method~\cite{Vaughan} was used, a stronger level of distribution $0.34096$ was required (see Remark~\ref{re:LZX} below for more information).

The paper is organized in the following way. In Section \ref{notation}, we recall some auxiliary results needed to carry out our work, specially those involving sums over primes and the linear sieve. In Section \ref{se:set-up}, we present and discuss the basic set-up and assumptions in our general sieving situation as well as apply the linear sieve to reduce our task of finding pairs $(p, P_k)$ to the task of finding a positive lower bound for the difference of two expressions $\Sigma_1$ and $\Sigma_2$. Toward the end of Section~\ref{se:set-up} we state Theorem~\ref{th:constLOD} mentioned above and provide Example \ref{ex:weights}, where we present three important specific choices of weights.

In Section \ref{lower}, we find a lower bound for $\Sigma_1$, whereas in Section \ref{upper} we find an upper bound for $\Sigma_2$. In Section \ref{se:Harman}, we illustrate the effect of the weights from Example~\ref{ex:weights} in the context of Theorem~\ref{th:HarmanImprove}, and prove Theorem~\ref{th:HarmanImprove} in Section~\ref{ssec:HarKuhn}. In Section~\ref{proof}, we prove Theorem~\ref{th:constLOD}. Finally, in Section \ref{se:U2gen} we present a way to deal with more general weights than those needed for our theorems.

All the numerical calculations have been carried out using Mathematica and are available alongside the arXiv version of this work.

\section{Notation and auxiliary results}\label{notation}
For $n \in \mathbb{N}$, we write $\Omega(n)$ for the total number of prime factors of $n$ and $\omega(n)$ for the number of distinct prime factors.

\subsection{Sums over primes}
\label{ssec:sumsoverprimes}
We will regularly utilize the standard device of using the prime number theorem to transform sums over primes to integrals. For convenience, we describe an example of this technique here. Indeed, letting $v \colon [0, 1] \to [1,2]$ be a Lipschitz function, we have by the prime number theorem that
\begin{align*}
&\sum_{\substack{x^{1/20} < p_1 < p_2 < p_3 < x^{1/2} \\ x^{1/4} < p_1 p_2 p_3 \leq x^{3/4} \\ p_1^2 p_2 \leq x^{2/3}}} \frac{v(\frac{\log p_1}{\log x})}{p_1 p_2 p_3 \log\frac{x}{p_1 p_2 p_3}} \\
&= (1+o(1)) \iiint\limits_{\substack{x^{1/20} < t_1 < t_2 < t_3 < x^{1/2} \\ x^{1/4} < t_1 t_2 t_3 \leq x^{3/4} \\ t_1^2 t_2 \leq x^{2/3}}} \frac{v(\frac{\log t_1}{\log x})}{(t_1 \log t_1) (t_2 \log t_2) (t_3 \log t_3) \log\frac{x}{t_1 t_2 t_3}} dt_1 dt_2 dt_3.
\end{align*}
Substituting $t_j = x^{\alpha_j}$, the above expression equals
\[
\frac{1+o(1)}{\log x} \iiint\limits_{\substack{1/20 < \alpha_1 < \alpha_2 < \alpha_3 < 1/2 \\ 1/4 < \alpha_1 + \alpha_2 + \alpha_3 \leq 3/4 \\ 2\alpha_1 + \alpha_2 \leq 2/3}} \frac{v(\alpha_1)}{\alpha_1 \alpha_2 \alpha_3 (1-\alpha_1 -\alpha_2 -\alpha_3)} d\alpha_1 d\alpha_2 d\alpha_3.
\]

\subsection{Sums over rough numbers}
\label{ssec:rough}
We shall need asymptotic formulas for sums over rough numbers. Define first the \textit{Buchstab function} $\omega_B \colon \mathbb{R}_{\geq 0} \to \mathbb{R}_{\geq 0}$ by $\omega_B(u) = 0$ if $u \in [0, 1)$, $\omega_B(u) = 1/u$ if $u \in [1, 2]$ and $(u \omega_B(u))' = \omega_B(u-1)$ if $u > 2$. 
\begin{lem}
\label{le:Buchstab}
Let $\varepsilon > 0$, $x \geq 2$ and $z \in [x^\varepsilon, x^{1-\varepsilon}]$. Then
\[
\sum_{\substack{n \leq x \\ (n, P(z)) = 1}} 1 = \frac{x}{\log z} \omega_B\left(\frac{\log x}{\log z}\right) + O\left(\frac{x}{(\log x)^2}\right).
\]
\end{lem}
\begin{proof}
This follows from an inductive argument and the prime number theorem, see~\cite[Section 1.4]{HarmanBook}.
\end{proof}

For $j \in \mathbb{N}$, we define the functions $c_j \colon \mathbb{R}_{\geq 0} \to \mathbb{R}_{\geq 0}$ by recurrence as $c_1(t) = \mathbf{1}_{t \geq 1}$ and
\begin{equation*}
c_j(t)= \int_{j}^{\max\{j, t\}}\frac{c_{j-1}(\alpha-1)}{\alpha-1}d\alpha.
\end{equation*}
Similarly to the proof of Lemma~\ref{le:Buchstab}, we use the prime number theorem and induction to establish the following result, which is also stated in~\cite[Section 2]{Zhao} and~\cite[Lemma 2.2]{BrudernKawada}. (Our Lemma~\ref{le:cj} below corresponds to the special case $\beta=0$ and $q=a=1$ of~\cite[Lemma 2.2]{BrudernKawada}). For completeness we sketch its proof.
\begin{lem}
\label{le:cj}
Let $\varepsilon > 0, j \in \mathbb{N}, x \geq 2$, and $z \in [x^\varepsilon, x/\log x]$. Then
\[
\sum_{\substack{n \leq x \\ (n, P(z)) = 1 \\ \Omega(n) = j}} 1 = \frac{x}{\log x} c_j\left(\frac{\log x}{\log z}\right) + O\left(\frac{x \log \log x}{(\log x)^2}\right).
\]
\end{lem}

\begin{proof}[Sketch of the proof]
In case $\frac{\log x}{\log z} \leq j$, the claim is trivial. In case $\frac{\log x}{\log z} > j,$ we induct on $j$. When $j = 1$, the claim follows immediately from the prime number theorem. Let us now assume that the claim holds for $j-1 \geq 1$ and establish it for $j$. We have
\begin{equation}
\label{eq:cjdec}
\sum_{\substack{n \leq x \\ (n, P(z)) = 1 \\ \Omega(n) = j}} 1 = \sum_{z < p \leq x^{1/j}} \sum_{\substack{n \leq x/p \\ (n, P(p)) = 1 \\ \Omega(n) = j-1}} 1 + O\left(\frac{x}{z}\right).
\end{equation}
The contribution of the sum over $p \in [\frac{x^{1/j}}{\log x}, x^{1/j}]$ to the right hand side of~\eqref{eq:cjdec} is $O(x\frac{\log \log x}{(\log x)^2})$. In the remaining range we can apply the induction hypothesis to the innermost sum. Thus we obtain
\begin{equation*}
\sum_{\substack{n \leq x \\ (n, P(z)) = 1 \\ \Omega(n) = j}} 1 = \sum_{z < p \leq \frac{x^{1/j}}{\log x}} \frac{x}{p\log(x/p)} c_{j-1}\left(\frac{\log(x/p)}{\log p}\right) + O\left(x\frac{\log \log x}{(\log x)^2}\right).
\end{equation*}
The claim follows by first using the prime number theorem to transform the sum over primes to an integral over $t$ (similarly to Section~\ref{ssec:sumsoverprimes}), and then substituting $t = x^{1/\alpha}$.
\end{proof}

We are actually interested in rough numbers with at least a given number of prime factors. In light of the above, for $J \in \mathbb{N}$, we define the function $C_J \colon \mathbb{R}_{\geq 1} \to \mathbb{R}_{\geq 0}$ via
\[
C_J(t) = t\omega_B(t) - \sum_{j = 1}^{J-1} c_j(t).
\]
On combining Lemmas~\ref{le:Buchstab} and~\ref{le:cj} we obtain the following.
\begin{lem}
\label{le:CJ}
Let $\varepsilon > 0, J \in \mathbb{N}, x \geq 2$ and $z \in [x^\varepsilon, x^{1-\varepsilon}]$. Then
\begin{equation*}
\sum_{\substack{\ell \leq x \\ (\ell, P(z)) = 1 \\ \Omega(\ell) \geq J}} 1 = \frac{x}{\log x} C_J\left(\frac{\log x}{\log z}\right)  + O\left(\frac{x}{(\log x)^2}\right).
\end{equation*}
\end{lem}

\subsection{Level of distribution and linear sieve}

We define now the \textit{level of distribution} in case of a linear sieving problem.

\begin{defn}[Level of distribution of $v_n$]
\label{def:lod}
Let $X \geq 1, x \geq 2,$ and $\theta \in (0, 1)$. Let $(v_n)_{n\leq x}$ be a sequence of complex numbers and let $g\colon \mathbb{N}\to \mathbb{R}_{\geq 0}$ be a multiplicative function such that $g(p) < p$ for every prime $p$. 

We say that the sequence $(v_n)_{n \leq x}$ has \emph{level of distribution $\theta$ with size $X$ and density $g$} if
\begin{align*}
\sum_{\substack{d \leq x^\theta}} \mu^2(d) \left|\sum_{\substack{n \leq x/d}} v_{dn} - \frac{g(d)}{d} X \right|  = O\left(\frac{X}{(\log x)^{100}}\right),
\end{align*}
and if we have, for all $z>w\geq 2$, 
\begin{align}
\label{eq:g(p)condition}
\prod_{\substack{w\leq p< z}}\left(1-\frac{g(p)}{p}\right)^{-1}\leq \left(1+O\left(\frac{1}{\log w}\right)\right)\frac{\log z}{\log w}.
\end{align}
\end{defn}

We extend the definition~\eqref{eq:SAzdef} to sequences, by writing, for $A=(v_n)_{n\leq x}$,
\[
S(A, z) = \sum_{\substack{n \leq x \\ (n, P(z)) = 1}} v_n.
\]
We recall now that the linear sieve functions $f,F:\mathbb{R}_{+}\to\mathbb{R}_{\geq 0}$ are the continuous functions defined through the system of differential equations
\begin{equation}
\label{eq:fFdef} 
\begin{cases}sF(s)=2e^\gamma & \text{if $0 < s\leq 3$};\\
sf(s)=0 & \text{if $0 < s \leq 2$};
\end{cases}\qquad
\begin{cases}(sF(s))' = f(s-1) & \text{if $s > 3$;}\\
(sf(s))'=F(s-1) & \text{if $s>2$.}
\end{cases}
\end{equation}
Note that 
\begin{align}\label{elementaryf}
f(s)&=\frac{2e^\gamma}{s}\log(s-1)&&\text{ if }2\leq s\leq 4.
\end{align}
Now we are ready to state the \textit{linear sieve}.
\begin{lem}[The linear sieve] \label{linearsieve} Let $x \geq z \geq 2$. Suppose that the sequence $A=(v_n)_{n\leq x}$ has level of distribution $\theta \in (0, 1)$ with density $g \colon \mathbb{N} \to \mathbb{R}_{\geq 0}$ and size $X \geq 1$. Write
\[
s = \frac{\log x^\theta}{\log z}.
\]
Then
\begin{equation*}
S(A, z) \geq \left(f(s)+O((\log x)^{-1/6})\right) \cdot X \prod_{p < z}\left(1-\frac{g(p)}{p}\right) 
\end{equation*}
and
\begin{equation*}
 S(A,z)\leq \left(F(s)+O((\log x)^{-1/6})\right) \cdot X \prod_{p < z}\left(1-\frac{g(p)}{p}\right).
\end{equation*}
Here the implied constants do not depend on the sequence $A$ itself but they may depend on $\theta$ and the implied constants in Definition~\ref{def:lod}.
\end{lem}
\begin{proof} See for instance~\cite[(12.12), (12.13)]{Opera}. The lower bound there is stated only for $s \geq 2$ but it is trivial for $s \in (0, 2)$. The upper bound is only stated for $s > 1$ but, when $s \in (0, 1]$ we can estimate $S(A, z) = S(A, x^{\theta/s}) \leq S(A, x^{\theta/2})$, use the linear sieve with $s = 2$, and then reach the conclusion by using~\eqref{eq:g(p)condition} and~\eqref{eq:fFdef}.
\end{proof}

\section{Set-up and assumptions} \label{se:set-up}
Let us now formalize the sieve set-up that we will consider. Let $x \geq 2$ be sufficiently large, let $S \in \mathbb{N}$, and let $\mathcal{A}\subseteq ([1,x]\cap\mathbb{Z})^2$ be such that we wish to find pairs $(p, P_S) \in \mathcal{A}$. For instance, in the case of the twin prime problem, we may take
\[
\mathcal{A} = \{(n, n+2) \in \mathbb{Z}^2 \colon n \leq x-2\}
\]
and, in the case of Theorem~\ref{th:HarmanImprove}, we could take
\[
\mathcal{A} = \{(n_1, n_2) \in \mathbb{Z}^2 \colon n_1, n_2 \leq x, |\lambda_0 + \lambda_1 n_1 + \lambda_2 n_2| < n_1^{-\rho}\},
\]
although, for technical reasons, our choice in Section~\ref{se:Harman} will be slightly different.

Next we discuss how we apply a weighted sieve with switching to the set $\mathcal{A}$. Let 
\begin{equation}
\label{eq:andef}
a_n = \sum_{\substack{p \leq x \\ (p, n) \in \mathcal{A}}} 1,
\end{equation}
and let $v \geq u > 2$. We consider weights $W: \mathbb{N} \to \mathbb{R}$ of the form
\begin{equation}\label{weight}
W(n) := 1-\sum_{\substack{p|n\\x^{1/v}\leq p<x^{1/u}}}w_p,
\end{equation}
with coefficients $w_p \in [0, 1]$. In fact, we will take $w_p = w\left(\frac{\log p}{\log x}\right)$ for a Lipschitz function $w \colon \mathbb{R}_{\geq 0} \to [0, 1]$ such that $w(\alpha) = 0$ for $\alpha \not \in [1/v, 1/u]$, and write $C_w > 0$ for a fixed Lipschitz constant of the function $w$. 

We immediately see that there exists some $(p,P_S) \in \mathcal{A}$ if 
\[
\sum_{\substack{n\leq x\\(n,P(x^{1/v}))=1}}a_n W(n) - \sum_{\substack{n\leq x\\(n,P(x^{1/v}))=1 \\ \Omega(n) \geq S+1}}a_n W(n) > 0.
\]
 In the above second sum, we may estimate the contribution of $n$ with $W(n) < 0$ trivially, and so, writing here and later $W^+(n) = \max\{W(n), 0\}$, it suffices to show that
\begin{equation}
\label{eq:posclaim}
\Sigma_1 - \Sigma_2' > 0,
\end{equation}
where
\begin{equation}
\label{eq:S1S2def}
\Sigma_1 := \sum_{\substack{n\leq x\\(n,P(x^{1/v}))=1}}a_n W(n) \quad \text{and} \quad \Sigma_2' := \sum_{\substack{n\leq x\\(n,P(x^{1/v}))=1 \\ \Omega(n) \geq S+1}}a_n W^+(n).
\end{equation}

Before stating the formal hypotheses (see Assumption~\ref{as:All} below), we discuss informally the assumptions we shall need. Thus (Ai*), $i\in\{1,...,5\}$, will denote an informal assumption and then (Ai) will be its corresponding formal assumption.

We shall apply the linear sieve (Lemma \ref{linearsieve}) to find a lower bound for $\Sigma_1$. For this, we assume
\begin{itemize}
\item[(A1*)] The sequence $a_n$ has level of distribution $\theta_1$.
\end{itemize}

One typically chooses the weight $W$ in such a way that $\Sigma_2'$ does not count too many numbers. In particular $W(n) \mathbf{1}_{(n, P(x^{1/v})) = 1}$ will be positive only when $n$ does not have too many prime factors. We will assume
\begin{itemize}
\item[(A2*)] The weight is non-positive when $\omega(n) > R$.
\end{itemize}
As we will work with numbers with $(n, P(x^{1/v})) = 1$, this always holds for $R = \lceil v-1 \rceil$.

We write now
\begin{equation*}
\Sigma_2 := \sum_{\substack{n\leq x\\(n,P(x^{1/v}))=1 \\ S+1 \leq \Omega(n) \leq R}} |\mu(n)| a_n W^+(n).
\end{equation*}
Note that when Assumption (A2*) holds, the sums in the definitions of $\Sigma_2'$ and $\Sigma_2$ differ only for non-square-free $n$. Since $(n, P(x^{1/v})) = 1$ is a condition on those sums, such $n$ form a very sparse set, and the contribution is typically negligible. Hence we assume
\begin{itemize}
\item[(A3*)] Replacing $\Sigma_2'$ by $\Sigma_2$ introduces a negligible error.
\end{itemize}

With respect to $\Sigma_2$, we switch the roles of the variables before applying the linear sieve and also write $n = p_1m$, with $p_1$ being the smallest prime factor of $n$. We take out the smallest prime factor because in some cases the level of distribution of the switched problem depends on the size of $p_1$. Thus
\begin{equation}
\label{eq:Sig2split}
\Sigma_2 = \sum_{\substack{(p, n) \in \mathcal{A} \\ n \leq x \\ (n,P(x^{1/v}))=1 \\ S+1 \leq \Omega(n) \leq R}}|\mu(n)| W^+(n) = \sum_{\substack{(\ell, p_1 m) \in \mathcal{A} \\ p_1 m \leq x \\ p_1 \geq x^{1/v} \\ p \mid m \implies p > p_1 \\ S \leq \Omega(m) \leq R-1}} |\mu(m)| W^+(p_1 m) \mathbf{1}_{\ell \in \mathbb{P}}.
\end{equation}
We shall sieve the variable $\ell$. We split dyadically according to the size of $p_1$, and so, for any $P \in  [x^{1/v}, x^{1/(S+1)}]$, we define the sequence $(b_{P, \ell})_{\ell \leq x}$ via
\begin{equation}
\label{eq:bdef}
b_{P, \ell} = \sum_{\substack{p_1m \leq x \\ (\ell, p_1 m) \in \mathcal{A} \\ P < p_1 \leq 2P \\ p \mid m \implies p > p_1 \\ S \leq \Omega(m) \leq R-1}} \left(|\mu(m)| W_P^+(m) + \frac{C_w}{\log x}\right),
\end{equation} 
where
\begin{equation}
\label{eq:WPNdef}
W_P(m) := 1-w\left(\frac{\log P}{\log x}\right)-\sum_{\substack{p \mid m \\ x^{1/v} \leq p < x^{1/u}}} w\left(\frac{\log p}{\log x}\right).
\end{equation}
Note that since $C_w$ is a Lipschitz constant for the function $w$, we have, for every $p_1 \in (P, 2P]$ and every $m \leq x/p_1$ such that $p_1 \nmid m$,
\[
W(p_1 m) \leq W_P(m) + \frac{C_w}{\log x}.
\]

In order to apply the upper bound in the linear sieve (Lemma~\ref{linearsieve}) to the sequence $(b_{P, \ell})_{\ell\leq x}$, we assume
\begin{itemize}
\item[(A4*)] The sequence $(b_{x^\alpha, \ell})_{\ell \leq x}$ has level of distribution $\theta_2(\alpha)$.
\end{itemize}

The switching trick can be used in cases where we are able to perform sieving in both coordinates in $\mathcal{A}$ --- Assumption~(A\ref{as:lod}*) allows us to sieve the second coordinate whereas Assumption~(A\ref{as:swlod}*) allows us to sieve the first coordinate. In order to be able to compare the main terms coming from $\Sigma_1$ and $\Sigma_2$ we assume
\begin{itemize}
\item[(A5*)] The main terms coming from sieving $(a_n)_{n \leq x}$ and $(b_{P, \ell})_{\ell \leq x}$ are comparable in a natural way.
\end{itemize}

Let us now formalize all our assumptions.
\begin{assump}
\label{as:All}
Let $\delta > 0, S \in \mathbb{N}$ and $v \geq u \geq 2$ be fixed and such that $v \geq S+1$. Let also $x \geq 2$ be sufficiently large in terms of $\delta$ and the implied constants in Definition~\ref{def:lod}. We assume that $\mathcal{A} \subseteq ([1, x] \cap \mathbb{Z})^2$, 
\begin{align*}
\theta_1 &\in [1/u+\delta, 1), \quad \theta_2 \colon [1/v, 1/(S+1)] \to [\delta, 1), \\
X_1 &\geq 1, \quad X_2 \colon [1/v, 1/(S+1)] \to [1, \infty), \\
w &\colon \mathbb{R}_{\geq 0} \to [0, 1], \quad g_1, g_2 \colon  \mathbb{N} \to \mathbb{R}_{\geq 0}
\end{align*}
are such that $\theta_2$ and $w$ are Lipschitz functions, $w(\alpha) = 0$ for $\alpha \not \in [1/v, 1/u]$, and all of the following hold, with $C_w > 0$ a fixed Lipschitz constant for the function $w$,
\[
W(n) = 1-\sum_{\substack{p|n\\x^{1/v}\leq p<x^{1/u}}} w_p, \quad w_p = w\left(\frac{\log p}{\log x}\right),
\]
and $W_P(m)$ as in~\eqref{eq:WPNdef}.

\begin{enumerate}[({A}1)]
\item \label{as:lod} The sequence $(a_n)_{n\leq x}$ defined in~\eqref{eq:andef} has level of distribution $\theta_1$ with size $X_1$ and density $g_1$.
\item \label{as:R} We have $W(n)\mathbf{1}_{(n, P(x^{1/v})) = 1} \leq 0$ whenever $\omega(n) > R$.
\item \label{as:squares}
We have
\[
\sum_{p \geq x^{1/v}} \sum_{\substack{n \leq x/p^2 \\ (n, P(x^{1/v})) = 1}} a_{np^2} \leq \frac{X_1}{(\log x)^2}.
\]
\item \label{as:swlod} For every $\alpha \in [1/v, 1/(S+1)]$, the sequence $(b_{x^\alpha, \ell})_{\ell\leq x}$ defined in~\eqref{eq:bdef} has level of distribution $\theta_2(\alpha)$ with size $X_2(\alpha)$ and density $g_2$.
\item \label{as:Xrel}
For every $\alpha \in [1/v, 1/(S+1)]$, we have
\[
X_2(\alpha) = (1+o(1))\sum_{\substack{p_1m \leq x \\ x^\alpha < p_1\leq 2x^\alpha \\ p \mid m \implies p > p_1 \\ S \leq \Omega(m) \leq R-1}}\left(|\mu(m)| W_{x^\alpha}^+(m) + \frac{C_w}{\log x}\right) \cdot \frac{X_1}{\sum_{p \leq x} 1}.
\]
\end{enumerate}
\end{assump}
Now, we can state our general result about finding pairs $(p, P_3)$ when the level of distribution of the original problem is the same as that of the switched problem.
\begin{thm}
\label{th:constLOD}
Let $x \geq 2$ and $\mathcal{A}\subset([1,x]\cap\mathbb{Z})^2$ be such that Assumption~\ref{as:All} holds with $S=3, v = 20, u = 6, w(\alpha) = \frac{1}{2} \cdot \mathbf{1}_{\alpha \in [1/v, 1/u]}$,  $\theta_1 = \theta_2(\alpha) = 0.267$, some $g_1 = g_2$ and some $X_1, X_2$. Then
\[
\sum_{(p, P_3) \in \mathcal{A}} 1 \gg \frac{X_1}{\log x}.
\]
\end{thm}
The proof of Theorem~\ref{th:constLOD} will be provided in Section~\ref{proof}.
\begin{rem}
\label{re:LZX}
Li, Zhang, and Xue~\cite{LZX} studied Piatetski-Shapiro primes $p$ such that 
\begin{equation}
\label{eq:ChenP3}
p+2=P_3,
\end{equation}
showing that, when $\gamma \in (0.9989445, 1)$, there are infinitely many solutions to~\eqref{eq:ChenP3} with $p = \lfloor n^{1/\gamma} \rfloor$ for some $n \in \mathbb{N}$. Following Vaughan~\cite{Vaughan} they used Richert's sieve with switching, using a trivial upper bound for $W(n)$ in the switched term (similarly to Harman's argument that we discuss in Section~\ref{ssec:HarRich}). This way they required level of distribution $0.34096$ which is a much stronger assumption than we need in Theorem~\ref{th:constLOD}. Consequently, Theorem~\ref{th:constLOD} can be used to improve upon the result concerning Piatetski-Shapiro primes in~\cite{LZX}. However, the level of distribution in this problem decreases so quickly with $\gamma$ that the improvement would be very modest, replacing the range for $\gamma$ by $\gamma \in (0.996651, 1)$. 

It would probably be possible to make a small further improvement to the range of $\gamma$ using that the level of distribution of the switched problem is somewhat better for some parts of the sum, similarly to (A4) and Section~\ref{se:Harman} where the level of distribution of the switched sum depends on the smallest prime factor.
\end{rem}

Recall from~\eqref{eq:posclaim} that in order to find pairs $(p, P_S)$ it suffices to show that $\Sigma_1 - \Sigma'_2 > 0$. Applying Assumptions (A2) and (A3) we see that in fact it suffices to show that
\begin{equation}
\label{eq:posclaims1s2}
\Sigma_1 - \Sigma_2 \gg \frac{X_1}{\log x}.
\end{equation}
Before turning to lower bounding $\Sigma_1$ and upper bounding $\Sigma_2$, let us discuss some choices for $w$ that can be found from the literature.

\begin{examp}
\label{ex:weights}
\begin{enumerate}[(i)] 
\item \label{it:triv} Taking $w(\alpha) = 0$ and $u = v$, we obtain the trivial weights $W_{\text{Trivial}}(n) = 1$.  Hence (A\ref{as:R}) holds with $R = \lceil v-1 \rceil$. Trivial weights with switching were used at least in~\cite{Kawada} and~\cite{Zhao}.
\item \label{it:Kuhn} Taking $v > u$ and $w(\alpha) = \frac{1}{2} \cdot \mathbf{1}_{\alpha \in [1/v, 1/u]}$, we obtain the Kuhn type weights
\begin{equation*}
W_{\text{Kuhn}}(n)=1\ -\sum_{\substack{p|n\\x^{1/v}\leq p<x^{1/u}}}\frac{1}{2}.
\end{equation*}
Here $W_{\text{Kuhn}}(n)\mathbf{1}_{(n, P(x^{1/v})) = 1} \leq 0$ unless $n$ has all its prime factors  greater or equal than $x^{1/v}$ and at most one distinct prime factor in $ [x^{1/v}, x^{1/u})$. Hence (A\ref{as:R}) holds with $R =  \lceil u(1-1/v) \rceil$. Kuhn type weights with switching have been used for detecting $P_2$ numbers, starting from the work of Chen~\cite{Chen}.
\item \label{it:Rich} Taking $w(\alpha) = \lambda(1-u\alpha) \cdot \mathbf{1}_{\alpha \in [1/v, 1/u]}$ for some $\lambda > 0$, we obtain Richert's weights
\begin{equation*}
W_{\text{Richert}}(n)=1-\lambda \sum_{\substack{p|n\\x^{1/v}\leq p<x^{1/u}}}\left(1-u\frac{\log p}{\log x}\right).
\end{equation*} 
In order to determine $R$ notice that, for $n \leq x$ with $(n, P(x^{1/v})) = 1$, we have
\begin{align}
\label{eq:WRichup}
\begin{aligned}
W_{\text{Richert}}(n) \leq 1-\lambda\sum_{\substack{p|n}} \left(1-u\frac{\log p}{\log x}\right) &= 1-\lambda \left(\omega(n)-u\frac{\log n}{\log x}\right) \\
&\leq  1-\lambda(\omega(n)-u).
\end{aligned}
\end{align}
Hence Assumption~(A\ref{as:R}) holds for $R \geq \lceil\frac{1}{\lambda} + u - 1\rceil$. Switching with Richert's weights was first utilized by Vaughan~\cite{Vaughan} and then used for instance in~\cite{Harman83} and~\cite{LZX}.
\end{enumerate}
\end{examp}

From now on we assume that Assumption~\ref{as:All} holds for some choices of parameters. In the following two sections we shall lower bound $\Sigma_1$ and upper bound $\Sigma_2$ in order to be in a position to prove~\eqref{eq:posclaims1s2}.

\section{Lower bounding $\Sigma_1$}\label{lower}
In this section we find a lower bound for $\Sigma_1$. From~\eqref{eq:S1S2def} and~\eqref{weight} we have
\begin{align*}
\begin{aligned}
\Sigma_1 &= \sum_{\substack{n\leq x\\(n,P(x^{1/v}))=1}}a_n W(n) = \sum_{\substack{n\leq x\\(n,P(x^{1/v}))=1}}a_n\left(1-\sum_{\substack{p|n\\x^{1/v}\leq p<x^{1/u}}}w_p\right) \\
&=\sum_{\substack{n\leq x\\(n,P(x^{1/v}))=1}}a_n - \sum_{\substack{x^{1/v}\leq p<x^{1/u}}}w_p \sum_{\substack{m \leq x/p \\ (m, P(x^{1/v})) = 1}} a_{mp}.
\end{aligned}
\end{align*}

Similarly to \cite[Section 25.2]{Opera}, we can apply the lower bound in the linear sieve (Lemma~\ref{linearsieve}) to the first term with $s = \theta_1 v$ and the upper bound in the linear sieve to the inner sum of the second term with 
\begin{equation*}
s =\frac{\log \frac{x^{\theta_1}}{p}}{\log x^{1/v}} = \theta_1 v - v\frac{\log p}{\log x}.
\end{equation*}
This way, we obtain
\begin{align}\label{weightedsieve}
\begin{aligned} \Sigma_1 &\geq  X_1 \left(f(\theta_1 v)-\sum_{x^{1/v}\leq p<x^{1/u}}\frac{g_1(p)}{p} w\left(\frac{\log p}{\log x}\right) F\left(\theta_1 v -v \frac{\log p}{\log x}\right)\right) \\
& \qquad \qquad \cdot \prod_{p < x^{1/v}} \left(1-\frac{g_1(p)}{p}\right) + o\left(X_1 \prod_{p < x^{1/v}} \left(1-\frac{g_1(p)}{p}\right)\right).
\end{aligned}
\end{align}

Notice first that since $g_1$ satisfies~\eqref{eq:g(p)condition}, we have, for any $z>w \geq 2$,
\begin{align*}
\begin{aligned}
\sum_{w \leq p < z} \frac{g_1(p)}{p} &= \log\Bigg( \prod_{w \leq p < z} \exp\left(\frac{g_1(p)}{p}\right)\Bigg) \leq \log\left(\prod_{w \leq p < z} \left(1-\frac{g_1(p)}{p}\right)^{-1}\right) \\
&\leq \log \frac{\log z}{\log w} + O\left(\frac{1}{\log w}\right) = \sum_{w \leq p < z} \frac{1}{p} + O\left(\frac{1}{\log w}\right).
\end{aligned}
\end{align*}
Now since $w$ and $F$ are Lipschitz functions in the relevant range, covering $[x^{1/v}, x^{1/u}]$ by $\ll \frac{\log x}{\log \log x}$ intervals of the type $[y, y \log x]$, this implies that
\begin{align*}
&\sum_{x^{1/v}\leq p<x^{1/u}}\frac{g_1(p)}{p} w\left(\frac{\log p}{\log x}\right) F\left(\theta_1 v -v \frac{\log p}{\log x}\right) \\
&\leq \sum_{x^{1/v}\leq p<x^{1/u}}\frac{1}{p} w\left(\frac{\log p}{\log x}\right) F\left(\theta_1 v -v \frac{\log p}{\log x}\right) + O\left(\frac{1}{\log \log x}\right).
\end{align*}
Furthermore, arguing as in Section~\ref{ssec:sumsoverprimes}, we see that
\begin{align*}
&\sum_{x^{1/v}\leq p<x^{1/u}}\frac{1}{p} w\left(\frac{\log p}{\log x}\right) F\left(\theta_1 v -v \frac{\log p}{\log x}\right) \\
&= (1+o(1))\int_{1/v}^{1/u} \frac{w(\alpha)}{\alpha} F\left(v\left(\theta_1 - \alpha\right) \right) d\alpha.
\end{align*}
Hence, by recalling the estimation~\eqref{weightedsieve}, we derive
\begin{equation}
\label{eq:S1genlow}
\Sigma_1 \geq X_1 \left(f(\theta_1 v)-\int_{1/v}^{1/u} \frac{w(\alpha)}{\alpha} F\left(v\left(\theta_1 - \alpha\right) \right) d\alpha +o(1)\right) \prod_{p < x^{1/v}} \left(1-\frac{g_1(p)}{p}\right).
\end{equation}

\section{Upper bounding $\Sigma_2$}\label{upper}
We now turn our attention to $\Sigma_2$. By~\eqref{eq:Sig2split} and~\eqref{eq:bdef}
\[
\Sigma_2 \leq \sum_{\substack{x^{1/v} \leq P \leq x^{1/(S+1)} \\ P = 2^jx^{1/v}}} \sum_{\ell \leq x} b_{P, \ell} \mathbf{1}_{\ell \in \mathbb{P}}.
\]
Applying the upper bound of the linear sieve with $s = 2\theta_2(\frac{\log P}{\log x}) \in (0, 2)$ and using Assumptions~(A\ref{as:swlod}) and~(A\ref{as:Xrel}), we derive
\begin{align*}
\Sigma_2 &\leq (1+o(1))\sum_{\substack{x^{1/v} \leq P \leq x^{1/(S+1)} \\ P = 2^jx^{1/v}}} \frac{2 e^\gamma}{2\theta_2(\frac{\log P}{\log x})} X_2\left(\frac{\log P}{\log x}\right) \prod_{p < x^{1/2}} \left(1-\frac{g_2(p)}{p}\right) \\
&= e^\gamma(1+o(1)) \frac{X_1}{\sum_{p \leq x} 1} \prod_{p < x^{1/2}} \left(1-\frac{g_2(p)}{p}\right)\\
& \qquad \cdot \sum_{\substack{x^{1/v} \leq P \leq x^{1/(S+1)} \\ P = 2^jx^{1/v}}} \frac{1}{\theta_2(\frac{\log P}{\log x})} \sum_{\substack{p_1m \leq x \\ P < p_1 \leq 2P \\ p \mid m \implies p > p_1 \\ S \leq \Omega(m) \leq R-1}} \left(|\mu(m)| W_P^+(m)+\frac{C_w}{\log x}\right). 
\end{align*}
Since $\theta_2 \colon [1/v, 1/(S+1)] \to [\delta, 1)$ is Lipschitz we have, for $p_1 \in (P, 2P]$,
\begin{equation*}
\left|\frac{1}{\theta_2\left(\frac{\log P}{\log x}\right)} -\frac{1}{\theta_2\left(\frac{\log p_1}{\log x}\right)}\right| = \frac{\left|\theta_2\left(\frac{\log p_1}{\log x}\right) - \theta_2\left(\frac{\log P}{\log x}\right)\right|}{\theta_2\left(\frac{\log P}{\log x}\right) \theta_2\left(\frac{\log p_1}{\log x}\right)} = O\left( \frac{1}{\theta_2 \left(\frac{\log P}{\log x}\right)^2 \log x}\right).
\end{equation*}
Furthermore, since $w$ is Lipschitz, we have, for $p_1 \in (P, 2P]$, $W_P^+(m) = W^+(pm) + O(1/\log x)$. Hence we obtain
\begin{align}
\label{eq:S2genupp}
\begin{aligned}
\Sigma_2 &\leq e^\gamma \frac{X_1}{\sum_{p \leq x} 1} \prod_{p < x^{1/2}} \left(1-\frac{g_2(p)}{p}\right) U_2 + o\left(X_1 \prod_{p < x^{1/2}} \left(1-\frac{g_2(p)}{p}\right)\right),
\end{aligned}
\end{align}
where
\begin{align}
\label{eq:U_2def}
U_2 := \sum_{\substack{p_1 m \leq x \\ p_1 \geq x^{1/v} \\ p \mid m \implies p > p_1 \\ S \leq \Omega(m) \leq R-1}} |\mu(m)| \frac{W^+(p_1 m)}{\theta_2(\frac{\log p_1}{\log x})}.
\end{align}
Now our main task in upper bounding $\Sigma_2$ is understanding $U_2$.

The most convenient way to evaluate $U_2$ depends on $u, v,$ and $w$. In the following two subsections we study $U_2$ in two important cases, first when $R$ is rather small and second when all prime divisors of $m$ are at least $x^{1/u}$. In the first case we end up with $(R-1)$ -fold integrals that can be numerically calculated whereas in the second case we have $W^+(p_1 m) = W^+(p_1)$ and we can evaluate the sum over $m$ using Lemma~\ref{le:CJ}. In Section~\ref{se:U2gen}, we will evaluate $U_2$ in a way that is useful in some more general cases.

In Sections~\ref{k=1} and~\ref{se:U2gen}, it is important to make the following observation: for some weights $w_p$ it is helpful to write $m = k\ell$, where all prime factors of $k$ are strictly smaller than $x^{1/u}$ and all prime factors of $\ell$ are at least $x^{1/u}$. Noticing that $W(p_1 m) = W(p_1 k)$, we see that
\begin{equation}
\label{eq:U_2}
U_2 = \sum_{\substack{p_1 k \leq x \\ p_1 \geq x^{1/v} \\ p \mid k \implies p \in (p_1, x^{1/u})}} \frac{W^+(p_1 k)}{\theta_2(\frac{\log p_1}{\log x})} \sum_{\substack{\ell \leq x/(p_1 k) \\ p \mid \ell \implies p \geq \max\{x^{1/u}, p_1\} \\ S-\Omega(k) \leq \Omega(\ell) \leq R-1-\Omega(k)}} |\mu(p_1k\ell)|.
\end{equation}

\subsection{Evaluating $U_2$ when $R$ is small}\label{ssec:U2Rsmall}

 We will present a way to calculate $U_2$ that holds in general but is most efficient numerically when $R$ is small, as it involves the calculation of an $(R-1)$ -fold integral. It should be compared to the analysis of Section~\ref{se:U2gen}.

We write in~\eqref{eq:U_2def} $m = p_2 \dotsm p_J$ with $p_J > \dotsc > p_2 > p_1$ and $J \in \{S+1, \dotsc, R\}$, so that
\begin{align*}
U_2 &= \sum_{J = S+1}^{R} \sum_{\substack{p_1 \dotsm p_J \leq x \\ p_J > \dotsb > p_1 \geq x^{1/v}}} \frac{\left(1-\sum_{p \mid p_1 \dotsm p_J} w(\frac{\log p}{\log x})\right)^+}{\theta_2(\frac{\log p_1}{\log x})},
\end{align*}
where $(y)^+ = \max\{y, 0\}$. Rearranging, we obtain
\begin{align}
\label{eq:U2sumre}
U_2 &= \sum_{J = S+1}^{R} \sum_{\substack{p_1 \dotsm p_{J-1} \leq x/p_{J-1} \\  p_{J-1} > \dotsb > p_1 \geq x^{1/v}}} \frac{1}{\theta_2(\frac{\log p_1}{\log x})} \sum_{\substack{p_J \leq \frac{x}{p_1 \dotsm p_{J-1}} \\ p_{J} > p_{J-1}}} \left(1-\sum_{p \mid p_1 \dotsm p_J} w\left(\frac{\log p}{\log x}\right)\right)^+.
\end{align}

Let us show that, now that we have added the condition $p_1 \dotsm p_{J-1} \leq x/p_{J-1}$, removing the condition $p_{J} > p_{J-1}$ yields an error of size $O(x/(\log x)^2)$. Indeed by the prime number theorem and rearranging,
\begin{align*}
&\sum_{J = S+1}^{R} \sum_{\substack{p_1 \dotsm p_{J-1} \leq x/p_{J-1} \\  p_{J-1} > \dotsb > p_1 \geq x^{1/v}}} \sum_{\substack{p_J \leq \frac{x}{p_1 \dotsm p_{J-1}} \\ p_{J} \leq p_{J-1}}} 1 \ll \sum_{J = S+1}^{R} \sum_{\substack{p_1 \dotsm p_{J-2} \leq x/p_{J-2}^2 \\  p_{J-2} > \dotsb > p_1 \geq x^{1/v}}} \sum_{p_{J-1} \leq (\frac{x}{p_1 \dotsm p_{J-2}})^{\frac{1}{2}}} \frac{p_{J-1}}{\log x} \\
& \ll \sum_{J = S+1}^{R} \sum_{\substack{p_1 \dotsm p_{J-2} \leq x/p_{J-2}^2 \\  p_{J-2} > \dotsb > p_1 \geq x^{1/v}}} \frac{x}{p_1 \dotsm p_{J-2} (\log x)^2} \ll \frac{x}{(\log x)^2}.
\end{align*}

Observe also that $p_J \leq x/(p_1 \dotsm p_{J-1}\log x)$ makes a contribution $O(x/(\log x)^2)$ to $U_2$. On the other hand, when  $x/(p_1 \dotsm p_{J-1}\log x)< p_J\leq x/(p_1 \dotsm p_{J-1})$, since $w$ is Lipschitz, we may approximate 
\begin{equation}
\label{eq:wpJapprox}
w\left(\frac{\log p_J}{\log x}\right) = w\left(\frac{\log\frac{x}{p_1 \dotsm p_{J-1}}}{\log x}\right) +O\left(\frac{\log \log x}{\log x}\right).
\end{equation}

Hence, by using \eqref{eq:wpJapprox} and then the prime number theorem, the sum over $p_J$ in~\eqref{eq:U2sumre} (with the condition $p_J > p_{J-1}$ removed) equals 
\begin{align*}
\frac{x \cdot \left(1- \sum_{p \mid p_1 \dotsm p_{J-1}} w\left(\frac{\log p}{\log x}\right)-w\left(\frac{\log \frac{x}{p_1 \dotsm p_{J-1}}}{\log x}\right)\right)^+}{p_1 \dotsm p_{J-1} \log\frac{x}{p_1 \dotsm p_{J-1}}} + O\left(\frac{x \log \log x}{(\log x)^2 p_1 \dotsm p_{J-1}}\right).
\end{align*}
Substituting the above expression into~\eqref{eq:U2sumre} and arguing as in Section~\ref{ssec:sumsoverprimes} we obtain that
\begin{align}
\label{eq:U2Rsmall}
\begin{aligned}
U_2 &= \frac{x}{\log x}\sum_{J = S+1}^{R} \idotsint\limits_{\substack{\alpha_{J-1} > \dotsc > \alpha_1 > 1/v \\ \sum_{j=1}^{J-1}\alpha_j \leq 1-\alpha_{J-1}}} \frac{\left(1- \sum_{\alpha \in \{\alpha_1, \dotsc, \alpha_{J-1}, 1-\sum_{j=1}^{J-1}\alpha_j\}} w\left(\alpha\right)\right)^+}{\theta_2(\alpha_1) \alpha_1 \dotsm \alpha_{J-1} (1-\sum_{j=1}^{J-1}\alpha_j)} d\alpha_1 \dotsm d\alpha_{J-1} \\
& \qquad + O\left(\frac{x \log \log x}{(\log x)^2}\right).
\end{aligned}
\end{align}
We shall use this formula in Section~\ref{ssec:HarRich}.

\subsection{Evaluating $U_2$ when $k = 1$} \label{k=1}
Recall that in~\eqref{eq:U_2} we have $p_1 \geq x^{1/v}$ and all the prime factors of $k$ are from the interval $(p_1, x^{1/u})$. Therefore, in Example~\ref{ex:weights}\eqref{it:triv}--\eqref{it:Kuhn}, $W^+(p_1 k) \neq 0$ only when $k = 1$. In such case, Lemma~\ref{le:CJ} gives
\begin{align*}
U_2 &\leq \sum_{\substack{x^{1/v} \leq p_1 \leq x^{1/(S+1)}}} \frac{W^+(p_1)}{\theta_2(\frac{\log p_1}{\log x})} C_{S} \left(\frac{\log \frac{x}{p_1}}{\log \max\{x^{1/u}, p_1\}}\right) \frac{x}{p_1 \log\frac{x}{p_1}}  +O\left(\frac{x}{(\log x)^2}\right).
\end{align*}
Arguing as in Section~\ref{ssec:sumsoverprimes}, we see that
\begin{align*}
\begin{aligned}
U_2 &= \frac{x}{\log x} \int_{1/v}^{1/(S+1)} \frac{(1-w(\alpha)) C_{S}\left(\min\{u(1-\alpha), \frac{1-\alpha}{\alpha}\}\right)}{\theta_2(\alpha) \alpha(1-\alpha)} d\alpha  +O\left(\frac{x}{(\log x)^2}\right).
\end{aligned}
\end{align*}
In case $u \geq S+1$, we obtain
\begin{align}
\label{eq:U2dec2}
\begin{aligned}
U_2 &= \frac{x}{\log x}\left( \int_{1/v}^{1/u} \frac{(1-w(\alpha)) C_{S}(u(1-\alpha))}{\theta_2(\alpha) \alpha(1-\alpha)} d\alpha  + \int_{1/u}^{1/(S+1)} \frac{C_{S}\left(\frac{1-\alpha}{\alpha}\right)}{\theta_2(\alpha) \alpha(1-\alpha)} d\alpha\right)  \\
&\phantom{xxxxxxx}+O\left(\frac{x}{(\log x)^2}\right) .
\end{aligned}
\end{align}
We shall use this formula in Sections~\ref{ssec:HarTriv}, and~\ref{ssec:HarKuhn}.

\section{Case study: Diophantine approximation with a prime and an almost-prime}
\label{se:Harman}
In this section we demonstrate what different choices of weights give in the problem addressed in Theorem~\ref{th:HarmanImprove}. We will assume $\rho \in (0, 1/5)$ --- in any case larger $\rho$ are out of reach.

Following Harman~\cite[Section 2]{Harman83}, we first rationalize the problem. Recall we want to find infinitely many solutions to
\[
|\lambda_0 + \lambda_1 p + \lambda_2 P_3| < p^{-\rho}
\]
with $\lambda_0 \in \mathbb{R}$ and $\lambda_1, \lambda_2 \in \mathbb{R}\setminus\{0\}$ such that $\lambda_1/\lambda_2$ is negative and irrational. By dividing by $-\lambda_2$, we can clearly assume that $\lambda_1 > 0$ and $\lambda_2 = -1$ if we increase $\rho$ by $\varepsilon'>0$. Hence we look for solutions to 
\begin{equation}
\label{eq:lam2rem}
|\lambda_0 + \lambda_1 p - P_3| < p^{-\rho}.
\end{equation}

Moreover, let $a/q$ be a convergent to the continued fraction for $\lambda_1$ and let $X = q^{8/5}$. We assume that $q$ is large in terms of $\lambda_0, \lambda_1, 1/\lambda_1$.  We can write $\lambda_0 = \frac{b}{q} + \nu$ with $b \in \mathbb{Z}$ and $|\nu| < 1/q$. We set $x = \max\{X, \frac{aX+b}{q}\}$ and focus on finding $(p,P_3)\in([1,x]\cap\mathbb{Z})^2$ such that~\eqref{eq:lam2rem} holds. 

With our choices of parameters $|1/q+p/q^2|\ll X^{-\frac{1}{4}}$ for $p \leq X$ and thus it suffices to show that the number of solutions to
\[
\left|\frac{b}{q} + \frac{a}{q} p - P_3\right| < \frac{X^{-\rho}}{2}
\]
tends to infinity with $X$, where $p \leq X$ and $P_3 \leq \frac{aX+b}{q}$. Thereupon, we define
\[
\mathcal{A} := \left\{(n_1, n_2) \in (\mathbb{Z} \cap [1, x])^2 \colon \left|\frac{b}{q} + \frac{a}{q} n_1 - n_2 \right| < \frac{X^{-\rho}}{2}\right\}.
\]

Let us write now $\Vert y \Vert$ for the distance from $y$ to the nearest integer(s) and $[y]$ for the nearest integer to $y$ when $\Vert y \Vert \neq 1/2$. As our aim is to find a pair $(p,P_3)\in \mathcal{A}$, we set $S = 3$. Notice that 
\[
a_n = \sum_{(p, n) \in \mathcal{A}} 1 = 
\begin{cases}
1 & \text{if $n = \left[\frac{b+pa}{q}\right]$ for some $p \leq X$ with $\Vert\frac{b+pa}{q}\Vert < \frac{X^{-\rho}}{2}$;} \\
0 & \text{otherwise,} 
\end{cases}
\]
so that $a_n$ is the characteristic function of the set $\mathcal{A}$ introduced in~\cite[Section 2]{Harman83}. 

We use the letters $\eta$ and $\varepsilon$ for arbitrarily small positive numbers and allow implied constants to depend on them. Thus, \cite[Lemma 9]{Harman83} tells us that assumption (A\ref{as:lod}) holds with $\theta_1 = 1/3-\rho-\varepsilon, g_1 = 1$, and $X_1 = X^{-\rho} \sum_{p \leq X} 1$.

Let us next verify Assumption (A\ref{as:squares}). First note that by estimating trivially $a_{np^2} \leq 1$, we have
\begin{equation}
\label{eq:A3largep}
\sum_{p \geq x^{1/3}} \sum_{\substack{n \leq x/p^2 \\ (n, P(x^{1/v})) = 1}} a_{np^2} \leq \sum_{d \geq x^{1/3}} \frac{x}{d^2} \ll x^{1-\frac{1}{3}} \leq \frac{X_1}{2(\log x)^2},
\end{equation}
where we have used that $x\asymp X$ and that $\rho<1/5$. On the other hand
\begin{align*}
&\sum_{x^{1/v} \leq p < x^{1/3}} \sum_{\substack{n \leq x/p^2 \\ (n, P(x^{1/v})) = 1}} a_{np^2} \\
&\leq \sum_{x^{1/v} \leq p < x^{1/3}} \left|\left\{(m, n) \in ([1,x]\cap\mathbb{Z})^2  \colon \left|\frac{b}{q} + \frac{a}{q} m -p^2 n\right| < \frac{X^{-\rho}}{2}\right\}\right| \\
& \leq \sum_{x^{1/v} \leq p < x^{1/3}} \left|\left\{m \leq x \colon \left\Vert\frac{b}{qp^2} + \frac{a}{qp^2} m \right\Vert < \frac{X^{-\rho}}{2p^2}\right\}\right|.
\end{align*}
Splitting $b+am$ into residue classes (mod ${qp^2}$), we see that this is 
\begin{align*}
&\leq \sum_{x^{1/v} \leq p < x^{1/3}}\ \sum_{|k|\leq qp^2 \cdot \frac{X^{-\rho}}{2p^2} }\ \sum_{\substack{m\leq x\\b+am\equiv k\ (\text{mod } qp^2)}}1\\
&\leq \sum_{\substack{x^{1/v} \leq p < x^{1/3} \\ p \nmid a}} \left(qX^{-\rho} + 1\right) \left(\frac{x}{qp^2}+1\right) + \sum_{\substack{x^{1/v} \leq p < x^{1/3} \\ p \mid a}} \sum_{\substack{|k| \leq \frac{qX^{-\rho}}{2} \\ p \mid k-b}} \frac{x}{q} \\
&\leq 2\sum_{x^{1/v} \leq p < x^{1/3}} \left(\frac{x X^{-\rho}}{p^2} + q X^{-\rho}\right) + \sum_{\substack{x^{1/v} \leq p < x^{1/3} \\ p \mid a}} \left(\frac{qX^{-\rho}}{p}+1\right) \frac{x}{q}\\
&\leq x^{1-1/v}X^{-\rho} + qx^{1/3}X^{-\rho} + vx^{1-1/v}X^{-\rho} + v\frac{x}{q} \leq \frac{X_1}{2(\log x)^2}.
\end{align*}
Now Assumption (A\ref{as:squares}) follows by combining this with~\eqref{eq:A3largep}.

On the other hand, we write, for $\alpha \in [1/v, 1/(S+1)]$ and $n \leq \frac{aX+b}{q}$,
\[
v(\alpha)_n = 
\begin{cases}
|\mu(m)| W_{x^\alpha}^+(m) + \frac{C_w}{\log x} & 
\begin{aligned} &\text{if $n = p_1 m$ with $x^\alpha < p_1 \leq 2x^\alpha$}, \\ &\qquad \text{$p \mid m \implies p > p_1, S \leq \Omega(m) \leq R-1$;} 
\end{aligned} \\
0 & \text{otherwise}.
\end{cases}
\]
so that
\[
b_{x^\alpha, \ell} = \sum_{\substack{p_1m \leq x \\ (\ell, p_1m) \in \mathcal{A} \\ x^\alpha < p_1 \leq 2x^\alpha \\ p \mid m \implies p > p_1 \\ S \leq \Omega(m) \leq R-1}} \left(|\mu(m)| W^+_{x^\alpha}(m) + \frac{C_w}{\log x}\right) = \mathbf{1}_{\Vert (b+\ell a)/q \Vert < X^{-\rho}/2} \ v(\alpha)_{[(b+\ell a)/q]}.
\]
Now $b_{x^\alpha, \ell}$ is a slight variant of $\mathcal{A}^\ast(\alpha)$ in~\cite[Lemma 10]{Harman83} --- the difference is that instead of the set $\mathcal{N}(\alpha)$ in~\cite{Harman83}, we have a sequence $(v(\alpha)_n)$. However the structure is similarly bilinear, with support on numbers $n = p_1 m $ with $p_1 \in (X^\alpha, 2X^\alpha]$ and coefficients of the type 
\[
\nu(\alpha)_n = \nu(\alpha)_{p_1m} = \mathbf{1}_{p_1 \in \mathbb{P} \cap (X^\alpha, 2X^\alpha]} a_m
\]
with $a_m \ll 1$ (as in~\cite{Harman83}, there is additionally a mild cross-condition $p \mid m \implies p > p_1$ but it can be easily removed for instance by using Perron's formula, see e.g.~\cite[Section 3.2]{HarmanBook}). Hence, similarly to~\cite[Lemma 10]{Harman83}, we have that (A\ref{as:swlod}) holds with $\theta_2(\alpha) = \frac{1}{2}(1-\alpha)-\rho-2\eta$, $g_2 = 1$. Also,
\[
X_2(\alpha) = \frac{X^{-\rho} q}{a} \sum_{n \leq \frac{aX+b}{q}} v(\alpha)_n,
\]
and by the prime number theorem, we see that 
\[
\frac{q}{a} \sum_{n \leq \frac{aX+b}{q}} v(\alpha)_n=(1+o(1)) \sum_{n \leq X} v(\alpha)_n,
\]
since $X$ is large in terms of the size of $a/q$. Thus, we see that Assumption~(A\ref{as:Xrel}) holds.

Next, recall that our aim is to show that
\begin{equation}
\label{eq:Harposclaim}
\Sigma_1 - \Sigma_2 \gg \frac{X_1}{\log x}
\end{equation}
with $\rho<1/5$ as large as possible. In the following sections, we are going to experiment with different weights, inspired by Example \ref{ex:weights}.

Before dealing with \eqref{eq:Harposclaim}, observe that \eqref{eq:S1genlow} gives
\begin{equation}
\label{eq:S1Harman}
\Sigma_1 \geq \frac{X_1}{\log x} e^{-\gamma} v \left(f(\theta_1 v)-\int_{1/v}^{1/u} \frac{w(\alpha)}{\alpha} F\left(v\left(\theta_1 - \alpha\right) \right) d\alpha +o(1)\right),
\end{equation}
where we have used Mertens' third theorem $\prod_{p<x^{1/v}}\left(1-\frac{1}{p}\right)=\frac{e^{-\gamma}v}{\log x}(1+o(1))$.

\subsection{Richert's weights and Harman's argument}
\label{ssec:HarRich}
Harman~\cite{Harman83} used Richert's weights (Example~\ref{ex:weights}\eqref{it:Rich}) with parameters
\begin{equation}\label{parameters}
v = \frac{4}{\theta_1}, \quad u = \frac{1+\eta}{\theta_1}, \quad \lambda = \frac{1}{5-u-\eta},
\end{equation}
where $\eta > 0$ is small. With the above choice of parameters, we have $\theta_1v=4$ and $v(\theta_1-\alpha)\leq 3$ for $\alpha \in [1/v, 1/u]$. Thus, by applying \eqref{eq:fFdef} and~\eqref{elementaryf} and recalling that $w(\alpha) = \lambda(1-u\alpha) \cdot \mathbf{1}_{\alpha \in [1/v, 1/u]}$, we derive from \eqref{eq:S1Harman} that
\[
\Sigma_1 \geq \frac{2X_1}{\log x} \left(\frac{\log 3}{\theta_1}- \int_{1/v}^{1/u} \frac{\lambda(1-u\alpha)}{\alpha (\theta_1 - \alpha)} d\alpha +o(1)\right).
\]
According to Example~\ref{ex:weights}\eqref{it:Rich}, we can take $R = \lceil \frac{1}{\lambda} + u - 1\rceil = 4$. Thus, upon using~\eqref{eq:S2genupp} and~\eqref{eq:U2Rsmall}, we obtain
\begin{align}
\label{eq:HarS2ActUpp}
\begin{aligned}
\Sigma_2 &\leq \frac{2X_1}{\log x} \iiint\limits_{\substack{\alpha_3 > \alpha_2 > \alpha_1 > 1/v \\\alpha_1+\alpha_2 + 2\alpha_3 \leq 1}} \frac{\left(1-w(\alpha_1)-w(\alpha_2)-w(\alpha_3)-w(1-\sum_{j=1}^3\alpha_j)\right)^+}{\theta_2(\alpha_1) \alpha_1 \alpha_2 \alpha_3 (1-\sum_{j=1}^3\alpha_j)} d\alpha_1 d\alpha_{2} d\alpha_3 \\
& \phantom{xxxxxxxxxxxxxxx}+ o\left(\frac{X_1}{\log x}\right).
\end{aligned}
\end{align}

Instead of evaluating this integral, Harman~\cite[Section 3]{Harman83} bounded $\Sigma_2$ using a pointwise upper bound for $W_{\text{Richert}}(n)$, so his treatment of $\Sigma_2$ corresponds to the use of trivial weights (Example~\ref{ex:weights}\eqref{it:triv}). Harman obtained the pointwise upper bound by noticing that~ \eqref{eq:WRichup} implies that, when $n\leq x$ is square-free and such that $(n,P(x^{1/v}))=1$ and $\Omega(n) = 4$, we have 
\[
W_{\text{Richert}}(n) \leq 1-\lambda(\Omega(n)-u) = \lambda\left(\frac{1}{\lambda}-(4-u)\right) = \lambda(1-\eta).
\]
Harman~\cite[Section 3]{Harman83} used this pointwise upper bound to obtain alternatively
\begin{align}
\label{eq:HarS2OwnUpp}
\Sigma_2 &\leq \frac{2X_1}{\log x} \iiint\limits_{\substack{\alpha_3 > \alpha_2 > \alpha_1 > 1/v \\\alpha_1+\alpha_2 + 2\alpha_3 \leq 1}} \frac{\lambda}{\theta_2(\alpha_1) \alpha_1 \alpha_2 \alpha_3 (1-\sum_{j=1}^3\alpha_j)} d\alpha_1 d\alpha_{2} d\alpha_3 + O\left(\frac{X_1}{(\log x)^2}\right).
\end{align}
This way, he obtained that~\eqref{eq:Harposclaim} holds when $\rho = 1/300$ and in fact $\rho = 1/150$ would be still admissible. 

By using the formula~\eqref{eq:HarS2ActUpp} instead of~\eqref{eq:HarS2OwnUpp}, thus using Richert's weights instead of the trivial ones also for $\Sigma_2$, one obtains that~\eqref{eq:Harposclaim} holds when $\rho = 1/25$ keeping the choice of parameters \eqref{parameters}. It turns out to be better to choose
\begin{equation*}
v = 19.2, \quad u = 4.1, \quad \text{and} \quad \lambda = \frac{1}{5.5-u-\eta}.
\end{equation*}
With these choices, we have that $(S,R)=(3,5)$ instead of $(3,4)$, so that the upper bound for $\Sigma_2$ following from~\eqref{eq:S2genupp} and~\eqref{eq:U2Rsmall} contains now a four-dimensional integral, and we obtain $\rho = 0.075$.

\subsection{Trivial weights}
\label{ssec:HarTriv}
As per Example~\ref{ex:weights}\eqref{it:triv}, we have $u = v$ and $w(\alpha) = 0$, so that by~\eqref{eq:S1Harman}
\[
\Sigma_1 \geq \frac{X_1}{\log x} e^{-\gamma} v \left(f(\theta_1 v)+o(1)\right).
\]
Suppose now that $v> S+1=4$. Then, by using~\eqref{eq:S2genupp} and~\eqref{eq:U2dec2}, we derive
\[
\Sigma_2 \leq \frac{2X_1}{\log x} \int_{1/v}^{1/4} \frac{C_{3}\left(\frac{1-\alpha}{\alpha}\right)}{\theta_2(\alpha) \alpha(1-\alpha)} d\alpha  +o\left(\frac{X_1}{\log x}\right).
\]
Evaluating the integrals shows that~\eqref{eq:Harposclaim} holds when $\rho = 1/16$ and $v = 10.8$.

\subsection{Kuhn's weights and the proof of Theorem \ref{th:HarmanImprove}}
\label{ssec:HarKuhn}
As per Example~\ref{ex:weights}\eqref{it:Kuhn}, we have $w(\alpha) = \frac{1}{2} \cdot \mathbf{1}_{\alpha \in [1/v, 1/u]}$, so that by~\eqref{eq:S1Harman}
\[
\Sigma_1 \geq \frac{X_1}{\log x} e^{-\gamma} v \left(f(\theta_1 v)-\int_{1/v}^{1/u} \frac{1}{2\alpha} F\left(v\left(\theta_1 - \alpha\right) \right) d\alpha +o(1)\right).
\]
Further, by~\eqref{eq:S2genupp} and~\eqref{eq:U2dec2}, we have, when $u \geq S+1 = 4$,
\begin{align*}
\Sigma_2 &\leq \frac{2X_1}{\log x}\left(\int_{1/v}^{1/u} \frac{C_{3}(u(1-\alpha))}{2 \theta_2(\alpha) \alpha(1-\alpha)} d\alpha
+\int_{1/u}^{1/4} \frac{C_{3}\left(\frac{1-\alpha}{\alpha}\right)}{\theta_2(\alpha) \alpha(1-\alpha)} d\alpha  +o(1)\right).
\end{align*}

Evaluating the above two integrals and taking $u=6.6$ and $v = 23$ shows that~\eqref{eq:Harposclaim} holds when $\rho = 0.092$, which implies Theorem~\ref{th:HarmanImprove}.

\section{Proof of Theorem~\ref{th:constLOD}}\label{proof}
Write $\theta = 0.267$ and let parameters be as in Theorem~\ref{th:constLOD} (in particular we apply Kuhn's weights). Similarly to Section~\ref{ssec:HarKuhn}
\[
\Sigma_1 \geq \frac{X_1}{\log x} e^{-\gamma} v \left(f(\theta v)-\int_{1/v}^{1/u} \frac{1}{2\alpha} F\left(v\left(\theta - \alpha\right) \right) d\alpha +o(1)\right),
\]
and (since $u > S+1 = 4$),
\begin{align*}
\Sigma_2 \leq \frac{2X_1}{\log x}\left(\int_{1/v}^{1/u} \frac{C_{3}(u(1-\alpha))}{2 \theta \alpha(1-\alpha)} d\alpha
+\int_{1/u}^{1/4} \frac{C_{3}\left(\frac{1-\alpha}{\alpha}\right)}{\theta \alpha(1-\alpha)} d\alpha  +o(1)\right).
\end{align*}
Evaluating the integrals with $u = 6$ and $v = 20$ we see that~\eqref{eq:posclaims1s2} holds.

\section{Evaluating $U_2$ in general}
\label{se:U2gen}
We return once more to evaluating
\begin{align}
\label{eq:U2gen}
\begin{aligned}
U_2 &= \sum_{\substack{p_1 m \leq x \\ p_1 \geq x^{1/v} \\ p \mid m \implies p > p_1 \\ S \leq \Omega(m) \leq R-1}} |\mu(m)| \frac{W^+(p_1 m)}{\theta_2(\frac{\log p_1}{\log x})} \\
&= \sum_{\substack{p_1 k \leq x \\ p_1 \geq x^{1/v} \\ p \mid k \implies p \in (p_1, x^{1/u})}} \frac{W^+(p_1 k)}{\theta_2(\frac{\log p_1}{\log x})} \sum_{\substack{\ell \leq x/(p_1 k) \\ p \mid \ell \implies p \geq \max\{x^{1/u}, p_1\} \\ S-\Omega(k) \leq \Omega(\ell) \leq R-1-\Omega(k)}} |\mu(p_1 k \ell)|.
\end{aligned}
\end{align}
The analysis in Section~\ref{ssec:U2Rsmall} was most efficient when $R$ is quite small. In Section~\ref{k=1} $R$ was allowed to be large as long as only $k = 1$ is contained in the sum. Here we consider the case when $k$ might have a few prime factors, and work under the following extra assumption.

\begin{itemize}
\item[(A6)] We have $W(n)\mathbf{1}_{(n, P(x^{1/v})) = 1} \leq 0$ whenever $n$ has more than $R_0$ prime factors in the interval $[x^{1/v}, x^{1/u})$.
\end{itemize}

Now Section~\ref{k=1} handles the cases $R_0 \in \{0, 1\}$ whereas the analysis in this section stays efficient as long as $R_0$ is rather small, which is often the case when we look for $P_S$ numbers with a small $S$.

Using Lemma~\ref{le:CJ} to the inner sum on the right hand side of~\eqref{eq:U2gen} and separating the contribution of $\ell = 1$, we have $U_2 \leq U_{2, 1} + U_{2, 2}$, where
\begin{align*}
U_{2,1}:=&\sum_{\substack{p_1 k \leq x \\ x^{1/v} \leq p_1 \leq x^{1/u} \\ p \mid k \implies p \in (p_1, x^{1/u}) \\ S \leq \Omega(k) \leq R_0 - 1}}  \frac{W^+(p_1 k)}{\theta_2(\frac{\log p_1}{\log x})} |\mu(p_1 k)| \\
U_{2,2}:=&\sum_{\substack{p_1 k \leq x \\ p_1 \geq x^{1/v} \\ p \mid k \implies p \in (p_1, x^{1/u}) \\\Omega(k) \leq R_0-1}} |\mu(p_1 k)| \frac{W^+(p_1 k)}{\theta_2(\frac{\log p_1}{\log x})} C_{S-\Omega(k)}\left(\frac{\log \frac{x}{p_1 k}}{\log \max\{x^{1/u}, p_1\}}\right) \frac{x}{p_1 k \log\frac{x}{p_1 k}}\\
&\ +\ O\Bigg( \sum_{\substack{p_1 k \leq x \\ p_1 \geq x^{1/v} \\ p \mid k \implies p \in (p_1, x^{1/u})}}\frac{x}{p_1 k (\log x)^2}\Bigg).
\end{align*}

In principle we could evaluate also the sums over $k$ by looking for a recursive asymptotic formula for the number of integers with all prime factors in a given interval, and using partial summation --- however, the outcome would look quite formidable. Instead, we recall that the functions $C_{S-\Omega(k)}$ are defined through recursion in Section~\ref{ssec:rough} and can be evaluated quite fast. With that numerical remark at hand, we may work out an asymptotic formula via splitting $U_{2, 1}$ and $U_{2, 2}$ according to the number of prime factors of $k$ and then using the prime number theorem.

We write $k = p_2 \dotsm p_J$ and obtain
\begin{align*}
U_{2, 1} &= \sum_{J = S+1}^{R_0} \sum_{\substack{p_1 \dotsm p_J \leq x \\ x^{1/v} \leq p_1 < \dotsc < p_J < x^{1/u}}} \frac{\left(1-\sum_{p \mid p_1 \dotsm p_J} w(\frac{\log p}{\log x})\right)^+}{\theta_2(\frac{\log p_1}{\log x})} \\
&= \sum_{J = S+1}^{R_0} \sum_{\substack{p_1 \dotsm p_{J-1} \leq x/p_{J-1} \\ x^{1/v} \leq p_1 < \dotsc < p_{J-1} < x^{1/u}}} \frac{1}{\theta_2(\frac{\log p_1}{\log x})} \sum_{\substack{p_J \leq \frac{x}{p_1 \dotsm p_{J-1}} \\ p_{J-1} < p_J < x^{1/u}}} \left(1-\sum_{p \mid p_1 \dotsm p_J} w\left(\frac{\log p}{\log x}\right)\right)^+.
\end{align*}

Here adding the condition $p_1 \dotsm p_{J-1} \geq x^{1-1/u}$ introduces an error of the size
\[
\ll \sum_{\substack{n \leq x^{1-1/u} \\ (n, P(x^{1/v})) = 1}} \sum_{p_J < x^{1/u}} 1 \ll \frac{x}{(\log x)^2}.
\]
Using similar arguments to Section~\ref{ssec:U2Rsmall} it is not difficult to see that we can, with a negligible error, also remove the condition $p_{J-1} < p_J$ and make the approximation~\eqref{eq:wpJapprox}.
Hence, $U_{2,1}$ equals
\begin{align*}
\sum_{J = S+1}^{R_0} \sum_{\substack{x^{1-1/u} \leq p_1 \dotsm p_{J-1} \leq x/p_{J-1} \\ x^{1/v} \leq p_1 < \dotsc < p_{J-1} < x^{1/u}}} \frac{\left(1- F(p_1,...,p_{J-1})\right)^+}{\theta_2(\frac{\log p_1}{\log x})} \sum_{\substack{p_J \leq \frac{x}{p_1 \dotsm p_{J-1}}}} 1+o\left(\frac{x}{\log x}\right),
\end{align*}
where
\begin{align*}
F(p_1,...,p_{J-1}):=\sum_{p \mid p_1 \dotsm p_{J-1}} w\left(\frac{\log p}{\log x}\right)+w\left(\frac{\log \frac{x}{p_1 \dotsm p_{J-1}}}{\log x}\right).
\end{align*}
Arguing as in Section~\ref{ssec:sumsoverprimes} we see that
\begin{equation*}
U_{2,1}=\frac{x}{\log x}(M_{1}+o(1)),
\end{equation*}
where $M_{1}$ equals
\begin{align*}
\sum_{J = S+1}^{R_0} \idotsint\limits_{\substack{\frac{1}{v} < \alpha_1 < \dotsc < \alpha_{J-1} < \frac{1}{u} \\ 1-\frac{1}{u} \leq \sum_{j=1}^{J-1}\alpha_j \leq 1-\alpha_{J-1}}} \frac{\left(1- \sum_{\alpha \in \{\alpha_1, \dotsc, \alpha_{J-1}, 1-\sum_{j=1}^{J-1}\alpha_j\}} w\left(\alpha\right)\right)^+}{\theta_2(\alpha_1) \alpha_1 \dotsm \alpha_{J-1} (1-\sum_{j=1}^{J-1}\alpha_j)} d\alpha_1 \dotsm d\alpha_{J-1}.
\end{align*}

In $U_{2, 2}$ we separate the case $p_1 \geq x^{1/u}$ (in which case $k = 1$) and argue similarly. We obtain
\begin{equation*}
U_{2,2}=\frac{x}{\log x}(M_{2}+o(1)),
\end{equation*}
where
\begin{align*}
& M_2 :=\int_{1/u}^{1/(S+1)} \frac{C_S\left(\frac{1}{\alpha}(1-\alpha)\right)}{\theta_2(\alpha_1) \alpha_1 (1-\alpha_1)} \\ 
&+\sum_{J = 1}^{R_0} \idotsint\limits_{\substack{\frac{1}{v} < \alpha_1 < \dotsc < \alpha_{J} < \frac{1}{u} \\ \sum_{j=1}^J\alpha_j \leq 1-\frac{S-J}{u}}} \frac{\left(1- \sum_{\alpha \in \{\alpha_1, \dotsc, \alpha_{J}\}} w\left(\alpha\right)\right)^+  C_{S+1-J}(u(1-\sum_{j=1}^J\alpha_j))}{\theta_2(\alpha_1) \alpha_1 \dotsm \alpha_{J} (1-\sum_{j=1}^J\alpha_j)} d\alpha_1 \dotsm d\alpha_{J}.
\end{align*}
Note that our work in Section~\ref{k=1} corresponds to the case $R_0 = 1$ and, naturally, the result is the same in this case.

\section*{Acknowledgements}
The authors would like to thank the anonymous referee for very careful reading of the paper. Both authors were supported by the Finnish Centre of Excellence in Randomness and Structures (Research Council of Finland grants no. 346307 and 364214), and the first author was additionally supported by Research Council of Finland grant no. 333707.

\bibliographystyle{plain}
\bibliography{refs1}

\end{document}